\documentclass[12pt]{article}
\usepackage{booktabs}
\usepackage{caption}
\usepackage{mathrsfs}
\usepackage[fleqn]{amsmath}
\allowdisplaybreaks[4] 
\usepackage{amsfonts,amsthm,amssymb,mathrsfs,bbding}
\usepackage{graphics,multicol}
\usepackage{graphicx}
\usepackage{color}
\usepackage{enumerate}
\usepackage{caption}
\usepackage{cite}
\usepackage{latexsym,bm}
\usepackage{mathtools}
\evensidemargin=\oddsidemargin\topmargin=-1.5cm

\newtheorem{thm}{Theorem}[section]

\newtheorem{lem}{Lemma}[section]

\newtheorem{exam}{Example}

\theoremstyle{definition}

\addtocounter{section}{0}
\begin{document}
\title{Enumerating  Cayley (di-)graphs on dihedral groups\footnote{Supported
by National Natural Science Foundation of China (Nos. 11671344 and 11531011).}}
\author{{\small Xueyi Huang, \ \ Qiongxiang Huang\footnote{
Corresponding author.}\setcounter{footnote}{-1}\footnote{
\emph{E-mail:} huangxymath@gmail.com (X. Huang),  huangqx@xju.edu.cn (Q. Huang).}}\\[2mm]\scriptsize
College of Mathematics and Systems Science,
\scriptsize Xinjiang University, Urumqi, Xinjiang 830046, P. R. China}
\date{}
\maketitle
{\flushleft\large\bf Abstract}
Let $p$ be an odd prime, and $D_{2p}=\langle \tau,\sigma\mid \tau^p=\sigma^2=e,\sigma\tau\sigma=\tau^{-1}\rangle$ the dihedral group of order $2p$. In this paper, we provide the number of (connected) Cayley (di-)graphs  on $D_{2p}$ up to isomorphism by using the P\'{o}lya enumeration theorem.  In the process, we also enumerate (connected)  Cayley digraphs on $D_{2p}$ of out-degree $k$ up to isomorphism for each $k$.
\vspace{0.1cm}
\begin{flushleft}
\end{flushleft}

\section{Introduction}\label{s-1}
Let $G$ be a finite group, and let $S$ be a subset of $G$ such that $1\not\in S$. The \emph{Cayley digraph} on $G$ with respect to $S$, denoted by $\mathrm{Cay}(G,S)$, is the digraph with vertex set $G$ and with an arc from $g$ to $h$ if $hg^{-1}\in S$. If $S$ is symmetric, i.e., $S^{-1}=\{s^{-1}\mid s\in S\}=S$, then $hg^{-1}\in S$ if and only if $gh^{-1}\in S$,  and so $\mathrm{Cay}(G,S)$ can be viewed as an undirected graph, which is  called the \textit{Cayley graph} on $G$ with respect to $S$. In particular, if $G$ is a cyclic group, then the Cayley (di-)graph  $\mathrm{Cay}(G,S)$ is  called a \emph{circulant (di-)graph}.

Let $\mathrm{Cay}(G,S)$ be the Cayley digraph on $G$ with respect to $S$. Suppose that $\sigma\in \mathrm{Aut}(G)$, where $\mathrm{Aut}(G)$ is the \emph{automorphism group} of $G$. Let $T=\alpha(S)$. Then it is easily shown that $\alpha$ induces an isomorphism from $\mathrm{Cay}(G,S)$ to $\mathrm{Cay}(G,T)$. Such an isomorphism is called a \emph{Cayley isomorphism}. However, it is possible for two Cayley digraphs $\mathrm{Cay}(G,S)$ and $\mathrm{Cay}(G,T)$ to be isomorphic but no Cayley isomorphisms mapping $S$ to $T$. The Cayley digraph (resp. Cayley graph) $\mathrm{Cay}(G,S)$ is called a \emph{DCI-graph} (resp. \textit{CI-graph}) of $G$ if, for any Cayley digraph (resp. Cayley graph) $\mathrm{Cay}(G,T)$, whenever $\mathrm{Cay}(G,S)\cong \mathrm{Cay}(G,T)$ we have $\alpha(S)=T$ for some $\alpha\in \mathrm{Aut}(G)$. A group $G$ is called a \emph{DCI-group} (resp. \textit{CI-group}) if all Cayley digraphs (Cayley graphs) on $G$ are DCI-graphs  (CI-graphs). Clearly, a DCI-group is always a CI-group, but the converse is not always right. The investigation of CI-graphs (DCI-graphs) or CI-groups (DCI-groups) stems from a conjecture proposed by \'{A}d\'{a}m\cite{Adam}: all circulant graphs are CI-graphs of the corresponding cyclic groups. This conjecture was disproved by Elspas and Turner \cite{Elspas}, and however, the conjecture caused a lot of activity on the characterization of CI-graphs (DCI-graphs) or CI-groups (DCI-groups) \cite{Babai1,Dobson,Dobson1,Dobson2,Dobson3,Dobson4,Huang,Huang1,Huang2,Hirasaka,Kovacs,Li2,Morris,Muzychuk,Muzychuk1,Somlai}. Another motivation for investigating CI-graphs (DCI-graphs) or CI-groups (DCI-groups) is to determine  and enumerate the isomorphic classes of Cayley graphs for a given group. By the definition, if $\mathrm{Cay}(G,S)$ is a CI-graph (DCI-graph), then to decide whether or not $\mathrm{Cay}(G,S)$ is isomorphic to $\mathrm{Cay}(G,T)$, we only need to decide whether or not there exists an automorphism $\alpha\in \mathrm{Aut}(G)$ such that $\alpha(S)=T$. For this reason, Mishna \cite{Mishna} (see also \cite{Alspach}) applied the P\'{o}lya enumeration theorem to count the isomorphic classes of  Cayley graphs (Cayley digraphs)  on some CI-groups (DCI-groups). Also, the isomorphic classes of some families of Cayley graphs which are edge-transitive but not arc-transitive were determined in \cite{Li1,Wang,Xu1}. For more results about CI-problem (DCI-problem) and determination for isomorphic classes of Cayley graphs, we refer the reader to the review paper \cite{Li} and references therein.

Let $p$ be an odd prime, and $D_{2p}=\langle \tau,\sigma\mid \tau^p=\sigma^2=e,\sigma\tau\sigma=\tau^{-1}\rangle$ the dihedral group of order $2p$. In \cite{Babai1}, Babai showed that $D_{2p}$ is a DCI-group. This remind us the  P\'{o}lya enumeration theorem could be used to enumerate the isomorphic classes of Cayley (di-)graphs of $D_{2p}$. In this paper, inspired by Mishna's work \cite{Mishna},  we obtain the number of (connected) Cayley (di-)graphs  on $D_{2p}$ up to isomorphism.  In the process, we also enumerate (connected)  Cayley digraphs on $D_{2p}$ of out-degree $k$ up to isomorphism for each $k$.

\section{Main tools}\label{s-2}
Let $G$ be a finite group and  $X$  a finite set. A (\textit{left}) \textit{action} of $G$ on $X$, denoted by $(G,X)$, is a map from $G\times X$ to $X$, with the image of $(g,x)$ being denoted by $gx$, which satisfies the following two conditions:
\begin{enumerate}[(1)]
\item $ex=x$ for all $x\in X$, where $e$ is the identity of $G$;
\item $(gh)x=g(hx)$ for all $g,h \in G$ and all $x \in X$. 
\end{enumerate}
If there is an action of $G$ on $X$, then we say that $G$ \textit{acts} on $X$ or that $X$ is a \textit{$G$-set}. The action $(G,X)$ is called \textit{faithful} if the only element $g\in G$ satisfying $gx=x$ for every $x\in X$ is the identity element. Let $X$ be a $G$-set. For each $x\in X$, we define the \textit{orbit} of $x$ to be the subset $Gx=\{gx\mid g\in G\}$ of $X$, and we define the \textit{stabilizer} of $x$ to be the subset $G_x=\{g\in G\mid gx=x\}$ of $G$.  It is easy to see that all the orbits form a partition of $X$ and that $G_x$ is a subgroup of $G$.

Let $G$ be a permutation group on $X$ ($|X|=n$). Then $G$  acts on $X$ naturally by defining $gx=g(x)$. For $g\in G$, we denote by $b_k(g)$ the number of cycles of length $k$ in the disjoint cycle decomposition of $g$, where $k=1,2,\ldots,n$. Then the \textit{cycle type} of $g\in G$ is defined as $\mathrm{type}(g)=(b_1(g),b_2(g),\ldots,b_n(g))$. Clearly, $b_1(g)+2b_2(g)+\cdots+nb_n(g)=n$. The \textit{cycle index} $\mathcal{I}(G,X)$ of the permutation group $G$ acting on $X$ is defined to be the polynomial
\begin{equation}\label{eq-1}
\mathcal{I}(G,X)=P_G(x_1,x_2,\ldots,x_n)=\frac{1}{|G|}\sum_{g\in G}{x_1^{b_1(g)}x_2^{b_2(g)}\cdots x_n^{b_n(g)}},
\end{equation}
where $x_1,x_2,\ldots,x_n$ are indeterminates. 

Let $A$ and $C$ be finite sets. Denote by 
\begin{equation}\label{eq-2}
C^A=\{f\mid f:A\rightarrow C\}
\end{equation}
the set of all maps from $A$ to $C$. Let $G$ be a permutation group acting on $A$. Then we obtain a group action $(G,C^A)$  naturally by setting:
\begin{equation}\label{eq-3}
gf=f\circ g^{-1}~~\mbox{for every $g\in G$ and $f\in C^A$},
\end{equation}
where $f\circ g^{-1}$ denotes the composite of the maps $f$ and $g^{-1}$. Under the group action $(G,C^A)$, two maps in $C^A$  are said to be \textit{$G$-equivalent} if they belong to the same orbit. The P\'{o}lya Enumeration Theorem provides  the number of orbits of the group action $(G,C^A)$. 

\begin{lem}(P\'{o}lya Enumeration Theorem, see \cite{Harary}, Chapter 2.)\label{polya-1}
Let $A$ and $C$ be finite sets with $|A|=n$ and $|C|=m$. Let $G$ be a permutation group acting on $A$. Denote by $\mathcal{F}$  the set  of all orbits of  the group action $(G,C^A)$. Then
\begin{equation}\label{eq-4}
|\mathcal{F}|=P_G(m,m,\ldots,m),
\end{equation}
where $P_G(x_1,x_2,\ldots,x_n)$ is the cycle index of $(G,A)$ defined in (\ref{eq-1}).
\end{lem}

Let $G$ be a permutation group acting on $A$.  Two $k$-subsets $S$ and $T$ of $A$ are said to be \textit{$G$-equivalent} if there exists some $g\in G$ such that $g(S)=T$. The following result enumerates the $G$-equivalent $k$-subsets of $A$.
\begin{lem}(See \cite{Harary}, Chapter 2.)\label{polya-2}
Let $A$  be a finite set with $|A|=n$, and let $G$ be a permutation group acting on $A$.  Then the number of $G$-equivalent classes of $k$-subsets of $A$ is equal to the coefficient of $x^k$ in the polynomial $P_G(1+x,1+x^2,\ldots,1+x^n)$, 
where $P_G(x_1,x_2,\ldots,x_n)$ is the cycle index of $(G,A)$ defined in (\ref{eq-1}).
\end{lem}
 
%

Recall that a group $G$ is  a DCI-group if all Cayley digraphs on $G$ are DCI-graphs. The following result due to Babai \cite{Babai1} shows that the dihedral group $D_{2p}$ ($p$ prime) is a DCI-group.
\begin{lem}(See \cite{Babai1}.)\label{DCI}
The two Cayley digraphs $\mathrm{Cay}(D_{2p},S)$ and $\mathrm{Cay}(D_{2p},T)$ on $D_{2p}$ ($p$ prime) are isomorphic
if and only if there exists some $\alpha\in \mathrm{Aut}(D_{2p})$ such that  $\alpha(S)=T$.
\end{lem}
In \cite{Rotmaler}, Rotmaler determined  the automorphism group of the dihedral group $D_{2n}$.
\begin{lem}(See \cite{Rotmaler}.)\label{Aut}
Suppose $n > 2$ is an integer and $D_{2n}=\langle\tau,\sigma\mid \tau^n=\sigma^2=e, \sigma\tau\sigma=\tau^{-1}\rangle=\{\tau^i,\tau^j\sigma\mid i,j \in \mathbb{Z}_n\}$ is the dihedral group of order $2n$. Then $\mathrm{Aut}(D_{2n})=\{\alpha_{s,t}\mid s\in \mathbb{Z}_n^\times, t\in \mathbb{Z}_n\}\cong \mathbb{Z}_n^\times\ltimes \mathbb{Z}_n$, where $\alpha_{s,t}(\tau^i)=\tau^{si}$ and $\alpha_{s,t}(\tau^j\sigma)=\tau^{sj+t}\sigma$ for all $i,j\in \mathbb{Z}_n$.
\end{lem}

\section{Enumerating Cayley digraphs on $D_{2p}$}\label{s-3}
Let $D_{2p}=\langle\tau,\sigma\mid \tau^p=\sigma^2=e, \sigma\tau\sigma=\tau^{-1}\rangle=\{\tau^i,\tau^j\sigma\mid i,j \in \mathbb{Z}_p\}$ be the dihedral group of order $2p$  ($p$ is an odd prime).  Take $A=D_{2p}\setminus\{e\}=\{\tau^i,\tau^j\sigma\mid i\in \mathbb{Z}_p\setminus\{0\}, j\in\mathbb{Z}_p\}$ and $C=\{0,1\}$. Then $\mathrm{Aut}(D_{2p})$ is a permutation group acting on $A$ and $C^A$. For $S\subseteq A$, let $f_S$ denote the characteristic function of $S$, that is, $f_S(a)=1$ if $a\in S$, and $f_S(a)=0$ if $a\in A\setminus S$. Clearly, $f_S\in C^A$ and $C^A$ consists of all characteristic  functions on $A$. By Lemma \ref{DCI}, we know that two Cayley digraphs  $\mathrm{Cay}(D_{2p},S)$ and $\mathrm{Cay}(D_{2p},T)$ on $D_{2p}$  are isomorphic if and only if there exists some $\alpha\in \mathrm{Aut}(D_{2p})$ such that  $\alpha(S)=T$, which is the case if and only if $f_S,f_T\in C^A$ are $\mathrm{Aut}(D_{2p})$-equivalent. Thus the number of Cayley digraphs on $D_{2p}$ up to isomorphism is equal to the number of orbits of the group action $(\mathrm{Aut}(D_{2p}),C^A)$. Therefore, by Lemma \ref{polya-1},
in order to enumerate Cayley digraphs on $D_{2p}$, we first need to compute the cycle index of the permutation group $\mathrm{Aut}(D_{2p})$ acting $A$. 

 By Lemma \ref{Aut}, we have $\mathrm{Aut}(D_{2p})=\{\alpha_{s,t}\mid s\in \mathbb{Z}_p^\times, t\in \mathbb{Z}_p\}$, where $\alpha_{s,t}(\tau^i)=\tau^{si}$ and $\alpha_{s,t}(\tau^j\sigma)=\tau^{sj+t}\sigma$ for all $i,j\in \mathbb{Z}_p$. Putting $A_1=\langle\tau\rangle\setminus\{e\}=\{\tau^i\mid i\in \mathbb{Z}_p\setminus\{0\}\}$ and $A_2=\langle\tau\rangle\sigma=\{\tau^j\sigma\mid j\in \mathbb{Z}_p\}$. Then $A=A_1\cup A_2$, and furthermore,  we observe that $\alpha_{s,t}(A_1)=A_1$ and $\alpha_{s,t}(A_2)=A_2$ for each $\alpha_{s,t}\in \mathrm{Aut}(D_{2p})$.
 
Also, since $p$ is an odd prime, we know that $\mathbb{Z}_p^\times=\{z\in \mathbb{Z}_p\mid \mathrm{gcd}(z,p)=1\}=\mathbb{Z}_p\setminus\{0\}$ is a cyclic group of order $p-1$ with the multiplication of integers module $p$. Thus we can assume that $\mathbb{Z}_p^\times=\langle z\rangle$ for some integer $z \in \mathbb{Z}_p^\times=\mathbb{Z}_p\setminus\{0\}$. Then for any $s\in\mathbb{Z}_p^\times$, there exists some $i_s\in\mathbb{Z}_{p-1}$ such that $s=z^{i_s}$. For example,  $i_s=0$ when $s=1$ and $i_s=\frac{p-1}{2}$ when $s=-1$. Furthermore, if  $s$ ranges over all elements of $\mathbb{Z}_p^\times$ then $i_s$ ranges over all elements of $\mathbb{Z}_{p-1}$. The following lemma is crucial to the calculation of cycle index.  

\begin{lem}\label{cycle-index-1}
Let $A=D_{2p}\setminus\{e\}$ and $\alpha_{s,t}\in \mathrm{Aut}(D_{2p})$ be defined as above. Let $z$ be a generating element of the cyclic group $\mathbb{Z}_p^\times$. Under the action of $\mathrm{Aut}(D_{2p})$ on $A$, the cycle type of $\alpha_{s,t}$ is given by  $type(\alpha_{s,t})=(b_1(\alpha_{s,t}),b_2(\alpha_{s,t}),\ldots,b_{2p-1}(\alpha_{s,t}))$, where 
\begin{equation}\label{eq-8}
b_k(\alpha_{1,t})=\left\{
\begin{aligned}
&2p-1&~&\mbox{if }k=1\mbox{ and } t=0,\\
&p-1&~&\mbox{if }k=1\mbox{ and } t\in\mathbb{Z}_p\setminus\{0\},\\
&1&~&\mbox{if } k=p \mbox{ and }t\in\mathbb{Z}_p\setminus\{0\},\\
&0&~&\mbox{otherwise},
\end{aligned}
\right.
\end{equation}
and for each $1\neq s=z^{i_s}\in\mathbb{Z}_p^\times$ (i.e., $i_s\neq 0$) and $t\in\mathbb{Z}_p$,
\begin{equation}\label{eq-9}
b_k(\alpha_{s,t})=b_k(\alpha_{s,0})=\left\{
\begin{aligned}
&1&~&\mbox{if }k=1,\\
&2\cdot\mathrm{gcd}(i_s,p-1)&~&\mbox{if } k=\frac{p-1}{\mathrm{gcd}(i_s,p-1)} ,\\
&0&~&\mbox{otherwise}.
\end{aligned}
\right.
\end{equation}
\end{lem}
\begin{proof}
Let $A_1$ and $A_2$ be defined as above. Since $\alpha_{s,t}(A_1)=A_1$ and $\alpha_{s,t}(A_2)=A_2$ for each $\alpha_{s,t}\in \mathrm{Aut}(D_{2p})$, we must have $b_{p+1}(\alpha_{s,t})=\dots=b_{2p-1}(\alpha_{s,t})=0$. For $\alpha_{s,t}\in \mathrm{Aut}(D_{2p})$, we consider the following two situations.

\textit{Case 1.} $s=1$;

Note that $\alpha_{1,t}(\tau^i)=\tau^i$ for each $i\in \mathbb{Z}_p\setminus\{0\}$, so the permutation $\alpha_{1,t}$ splits $A_1$ into $p-1$ cycles each of length $1$.  Also note that $\alpha_{1,t}(\tau^j\sigma)=\tau^{j+t}\sigma$ for $j\in \mathbb{Z}_p$. If $t=0$, then $\alpha_{1,0}(\tau^j\sigma)=\tau^j\sigma$ for each $j\in\mathbb{Z}_p$, and so $\alpha_{1,t}$ splits $A_2$ into $p$ cycles of length $1$. If $t\in\mathbb{Z}_p\setminus\{0\}$,  the order of $t$ in $\mathbb{Z}_p$ is equal to $o(t)=\frac{p}{\mathrm{gcd}(t,p)}=p$. Then,  for any $j\in \mathbb{Z}_p$, $\tau^j\sigma\in A_2$ is in the cycle $(\tau^j\sigma,\alpha_{1,t}(\tau^j\sigma),\alpha_{1,t}^2(\tau^j\sigma),\ldots,\alpha_{1,t}^{o(t)-1}(\tau^j\sigma))=(\tau^j\sigma,\tau^{j+t}\sigma,\tau^{j+2t}\sigma,\ldots,$ $\tau^{j+(p-1)t}\sigma)$. Thus the permutation $\alpha_{1,t}$ splits $A_2$ into exactly one cycle, which is  of length $p$. Therefore, we have obtain the cycle type of $\alpha_{1,t}$, as shown in (\ref{eq-8}).

\textit{Case 2.} $s\neq1$, say $s=z^{i_s}$ ($i_s\neq 0$).

Note that $\alpha_{s,t}(\tau^i)=\tau^{si}$ for $i\in \mathbb{Z}_p\setminus\{0\}$ and $\alpha_{s,t}(\tau^j\sigma)=\tau^{sj+t}\sigma$ for $j\in \mathbb{Z}_p$. Firstly, we shall prove that  $\alpha_{s,t}$ has the same cycle type as $\alpha_{s,0}$ for each $t\in\mathbb{Z}_p$. As $\alpha_{s,t}(\tau^i)=\tau^{si}=\alpha_{s,0}(\tau^i)$, $\alpha_{s,t}$ and $\alpha_{s,0}$ have the same cycle type in $A_1$. Now consider $\alpha_{s,t}$ and $\alpha_{s,0}$ acting on $A_2$. As $s\neq 1$, $(1-s)$ is invertible in $\mathbb{Z}_p^\times$,  we define a bijection in $A_2$, that is, $\beta: \tau^j\sigma\mapsto\tau^{j+(1-s)^{-1}t}\sigma$ for $j\in\mathbb{Z}_p$. Assume that $(\tau^{a_1}\sigma,\tau^{a_2}\sigma,\ldots,\tau^{a_r}\sigma)$ is a cycle of $\alpha_{s,0}$, i.e., $a_{\ell}=sa_{\ell-1}$ for $\ell\in \mathbb{Z}_r$. Then we have $\beta(\tau^{a_\ell}\sigma)=\tau^{a_\ell+(1-s)^{-1}t}\sigma$, and $\alpha_{s,t}(\beta(\tau^{a_{\ell-1}}\sigma))=\alpha_{s,t}(\tau^{a_{\ell-1}+(1-s)^{-1}t}\sigma)=\tau^{s(a_{\ell-1}+(1-s)^{-1}t)+t}\sigma=\tau^{a_{\ell}+s(1-s)^{-1}t+t}\sigma=\tau^{a_{\ell}+(s-1)(1-s)^{-1}t+(1-s)^{-1}t+t}\sigma=\tau^{a_{\ell}+(1-s)^{-1}t}\sigma=\beta(\tau^{a_\ell}\sigma)$ for each $\ell\in\mathbb{Z}_r$. Thus $(\beta(\tau^{a_1}\sigma),$ $\beta(\tau^{a_2}\sigma),\ldots,\beta(\tau^{a_r}\sigma))$ is a cycle of $\alpha_{s,t}$. Therefore, $\alpha_{s,t}$ and $\alpha_{s,0}$ also have the same cycle type in $A_2$ because $\beta$ is a bijection. Hence, we just need to consider the cycle type of $\alpha_{s,0}$ in $A=A_1\cup A_2$. Since the order of $s$ in $\mathbb{Z}_p^\times$ is equal to $o(s)=o(z^{i_s})=\frac{p-1}{\mathrm{gcd}(i_s,p-1)}$,  for any $i,j\in \mathbb{Z}_p\setminus\{0\}$, $\tau^i\in A_1$ is in the cycle $(\tau^i,\alpha_{s,0}(\tau^i),\alpha_{s,0}^2(\tau^i),\ldots,\alpha_{s,0}^{o(s)-1}(\tau^i))=(\tau^i,\tau^{si},\tau^{s^2i},\ldots,$ $\tau^{s^{o(s)-1}i})$ and
 $\tau^j\sigma\in A_2$ is in the cycle $(\tau^j\sigma,\alpha_{s,0}(\tau^j\sigma),\alpha_{s,0}^2(\tau^j\sigma),\ldots,$ $\alpha_{s,0}^{o(s)-1}(\tau^j\sigma))=(\tau^j\sigma,\tau^{sj}\sigma,\tau^{s^2j}\sigma,\ldots,\tau^{s^{o(s)-1}j}\sigma)$. Also note that $\alpha_{s,0}(\tau^0\sigma)=\tau^{0}\sigma$, so $\tau^0\sigma\in A_2$ is in the cycle $(\tau^0\sigma)$. Thus the permutation $\alpha_{s,0}$ splits $A_1$ into $\frac{p-1}{o(s)}=\mathrm{gcd}(i_s,p-1)$ cycles each of length $o(s)=\frac{p-1}{\mathrm{gcd}(i_s,p-1)}$, and splits $A_2$ into $\frac{p-1}{o(s)}=\mathrm{gcd}(i_s,p-1)$ cycles each of length $o(s)=\frac{p-1}{\mathrm{gcd}(i_s,p-1)}$ and one cycle of length $1$. Then we have obtained the cycle type of $\alpha_{s,t}$, as shown in (\ref{eq-9}).
 
 We complete the proof.
\end{proof}
Let $n$ be a positive integer. The \textit{Euler's totient function} $\Phi(n)$ is the number of integers $k$ in the range $1 \leq k \leq n$ for which the greatest common divisor $\mathrm{gcd}(n, k)$ is equal to $1$. Let $n=p_1^{k_1}\cdots p_s^{k_s}$ be the  prime factorization of $n$. Then the  \textit{Euler's product formula} states that $\Phi(n)=n\prod_{i=1}^{r}(1-\frac{1}{p_i})$. According to Lemma \ref{cycle-index-1}, we now give the cycle index of $\mathrm{Aut}(D_{2p})$ acting on $A=D_{2p}\setminus\{e\}$.
\begin{lem}\label{cycle-index-2}
The cycle index of $\mathrm{Aut}(D_{2p})$ acting on $A=D_{2p}\setminus\{e\}$ is given by
\begin{equation}\label{eq-10}
\mathcal{I}(\mathrm{Aut}(D_{2p}),A)=\frac{1}{p}x_1^{p-1}(x_p-x_1^p)+\frac{1}{p-1}x_1\cdot\sum_{d\mid(p-1)}\Phi(d)x_{d}^{\frac{2(p-1)}{d}},
\end{equation}
where $\Phi(\cdot)$ denotes  the Euler's totient function.
\end{lem}
\begin{proof}
By Lemma \ref{cycle-index-1}, the cycle index of $\mathrm{Aut}(D_{2p})$ acting on $A=D_{2p}\setminus\{e\}$ is 
\begin{eqnarray*}
&&\mathcal{I}(\mathrm{Aut}(D_{2p}),A)\\
&=&\frac{1}{|\mathrm{Aut}(D_{2p})|}\sum_{\alpha_{s,t}\in \mathrm{Aut}(D_{2p})}{x_1^{b_1(\alpha_{s,t})}x_2^{b_2(\alpha_{s,t})}\cdots x_{2p-1}^{b_{2p-1}(\alpha_{s,t})}}\\
&=&\frac{1}{p(p-1)}\sum_{s\in\mathbb{Z}_p^\times}\sum_{t\in\mathbb{Z}_p}{x_1^{b_1(\alpha_{s,t})}x_2^{b_2(\alpha_{s,t})}\cdots x_{2p-1}^{b_{2p-1}(\alpha_{s,t})}}\\
&=&\frac{1}{p(p-1)}\bigg[\sum_{t\in\mathbb{Z}_p}x_1^{b_1(\alpha_{1,t})}x_2^{b_2(\alpha_{1,t})}\cdots x_{2p-1}^{b_{2p-1}(\alpha_{1,t})}+\sum_{s\in\mathbb{Z}_p^\times\setminus\{1\}}\sum_{t\in\mathbb{Z}_p}x_1^{b_1(\alpha_{s,t})}x_2^{b_2(\alpha_{s,t})}\cdots x_{2p-1}^{b_{2p-1}(\alpha_{s,t})}\bigg]\\
&=&\frac{1}{p(p-1)}\bigg[\sum_{t\in\mathbb{Z}_p}x_1^{b_1(\alpha_{1,t})}x_2^{b_2(\alpha_{1,t})}\cdots x_{2p-1}^{b_{2p-1}(\alpha_{1,t})}+p\cdot\sum_{s\in\mathbb{Z}_p^\times\setminus\{1\}}x_1^{b_1(\alpha_{s,0})}x_2^{b_2(\alpha_{s,0})}\cdots x_{2p-1}^{b_{2p-1}(\alpha_{s,0})}\bigg]\\
&=&\frac{1}{p(p-1)}\bigg[x_1^{2p-1}+(p-1)x_1^{p-1}x_p+px_1\cdot\sum_{s=z^{i_s}\in\mathbb{Z}_p^\times\setminus\{1\}}x_{\frac{p-1}{\mathrm{gcd}(i_s,p-1)}}^{2\cdot\mathrm{gcd}(i_s,p-1)}\bigg]\\
&=&\frac{1}{p(p-1)}\bigg[x_1^{2p-1}+(p-1)x_1^{p-1}x_p+px_1\cdot\sum_{i_s\in\mathbb{Z}_{p-1}\setminus\{0\}}x_{\frac{p-1}{\mathrm{gcd}(i_s,p-1)}}^{2\cdot\mathrm{gcd}(i_s,p-1)}\bigg]\\
&=&\frac{1}{p(p-1)}\bigg[x_1^{2p-1}+(p-1)x_1^{p-1}x_p+px_1\cdot\sum_{\begin{smallmatrix}d\mid(p-1)\\d\neq 1\end{smallmatrix}}\Phi(d)x_{d}^{\frac{2(p-1)}{d}}\bigg]\\
&=&\frac{1}{p(p-1)}\bigg[x_1^{2p-1}+(p-1)x_1^{p-1}x_p+px_1\cdot\sum_{d\mid(p-1)}\Phi(d)x_{d}^{\frac{2(p-1)}{d}}-px_1^{2p-1}\bigg]\\
&=&\frac{1}{p}x_1^{p-1}(x_p-x_1^p)+\frac{1}{p-1}x_1\cdot\sum_{d\mid(p-1)}\Phi(d)x_{d}^{\frac{2(p-1)}{d}},
\end{eqnarray*}
where $\Phi(\cdot)$ denotes the Euler's totient function.
\end{proof}
According to  Lemmas \ref{cycle-index-2}, \ref{polya-1} and the arguments at the begining of the section, we  obtain the number of Cayley digraphs on $D_{2p}$ up to isomorphism immediately.
\begin{thm}\label{thm-enu-1}
Let $p$ be an odd prime. The number of Cayley digraphs on $D_{2p}$ up to isomorphism is equal to
\begin{equation}\label{eq-11}
\mathcal{N}=\frac{1}{p}2^{p}(1-2^{p-1})+\frac{2}{p-1}\cdot\sum_{d\mid(p-1)}\Phi(d)2^{\frac{2(p-1)}{d}},
\end{equation}
where $\Phi(\cdot)$ is the Euler's totient function.
\end{thm}
In \cite{Mishna}, Mishna enumerated the circulant digraphs of order $p$ up to isomorphism.
\begin{lem}(See \cite{Mishna}.)\label{circ-digraph}
Let $p$ be an odd prime. The number of circulant digraphs of order $p$ up to isomorphism is given by
\begin{equation}\label{eq-12}
\mathcal{N}_c=\frac{1}{p-1}\sum_{d\mid(p-1)}\Phi(d)2^{\frac{p-1}{d}},
\end{equation}
where $\Phi(\cdot)$ is the Euler's totient function.
\end{lem}
It is well known that a Cayley digraph $\mathrm{Cay}(G,S)$ is connected  if and only if $\langle S \rangle=G$. Thus, for $S\subseteq A=D_{2p}\setminus\{e\}$, the Cayley digraph $\mathrm{Cay}(D_{2p},S)$ is disconnected if and only if $S\subseteq A_1=\langle\tau\rangle\setminus\{e\}=\{\tau^i\mid i\in \mathbb{Z}_p\setminus\{0\}\}$ or $S=\{\tau^j\sigma\}\subseteq A_2=\langle\tau\rangle\sigma$ for  $j\in \mathbb{Z}_p$ because $p$ is a prime. Also note that $\mathrm{Cay}(D_{2p},\{\tau^j\sigma\})\cong \mathrm{Cay}(D_{2p},\{\sigma\})$ for each $j$ since $\alpha_{0,j}(\sigma)=\tau^j\sigma$. Hence, from Theorem \ref{thm-enu-1} and  Lemma \ref{circ-digraph} we have the following result immediately.
\begin{thm}\label{thm-enu-2}
Let $p$ be an odd prime. The number of connected Cayley digraphs on $D_{2p}$ up to isomorphism is equal to
\begin{equation}\label{eq-12-1}
\mathcal{N}_1=\mathcal{N}-\mathcal{N}_{c}-1,
\end{equation}
where $\mathcal{N}$ and $\mathcal{N}_{c}$ are presented in (\ref{eq-11}) and (\ref{eq-12}), respictively.
\end{thm}

\begin{thm}\label{thm-enu-3}
Let $p$ be an odd prime and let  $\mathcal{M}_{k}$ denote the number of  Cayley digraphs  on $D_{2p}$  of out-degree $k$ up to isomorphism. Then  
\begin{equation}\label{eq-12-2}
\mathcal{M}_k=\left\{
\begin{aligned}
&1&~&\mbox{if } k=0\mbox{ or } 2p-1,\\
&\frac{1}{p}\bigg[\binom{p-1}{k}-\binom{2p-1}{k}\bigg]+\mathcal{M}_k^1&~&\mbox{if } 1\leq k\leq p-1,\\
&\frac{1}{p}\bigg[\binom{p-1}{k-p}-\binom{2p-1}{k}\bigg] +\mathcal{M}_k^1&~& \mbox{if } p\leq k\leq 2p-2,\\
\end{aligned}
\right.
\end{equation}
where 
\begin{equation*}\mathcal{M}_{k}^1=\frac{1}{p-1}\bigg[\sum_{d\mid\mathrm{gcd}(p-1,k)}\Phi(d)\binom{\frac{2(p-1)}{d}}{\frac{k}{d}}+\sum_{d\mid\mathrm{gcd}(p-1,k-1)}\Phi(d)\binom{\frac{2(p-1)}{d}}{\frac{k-1}{d}}\bigg]\end{equation*} and $\Phi(\cdot)$ is the Euler's totient function.
\end{thm}
\begin{proof}
Clearly, if $k=0$ (resp. $k=2p-1$), then $\mathrm{Cay}(D_{2p},\emptyset)$ (resp.  $\mathrm{Cay}(D_{2p},D_{2p}\setminus\{e\})$) is the unique Cayley digraph on $D_{2p}$ with out-degree $k$. It suffices to consider $1\leq k\leq 2p-2$. By Lemma \ref{cycle-index-1}, the cycle index of $\mathrm{Aut}(D_{2p})$ acting on $A=D_{2p}\setminus\{0\}$ is 
\begin{equation*}
\mathcal{I}(\mathrm{Aut}(D_{2p}),A)=\frac{1}{p}x_1^{p-1}(x_p-x_1^p)+\frac{1}{p-1}x_1\sum_{d\mid(p-1)}\Phi(d)x_{d}^{\frac{2(p-1)}{d}}.
\end{equation*}
Putting $x_i=1+x^i$ in the above equation, we obtain the polynomial
\begin{equation*}
\begin{aligned}
Q(x)&=\frac{1}{p}((1+x)^{p-1}(1+x^p)-(1+x)^{2p-1}))+\frac{1}{p-1}(1+x)\sum_{d\mid(p-1)}\Phi(d)(1+x^d)^{\frac{2(p-1)}{d}}\\
&=\frac{1}{p}\bigg[\sum_{i=0}^{p-1}\binom{p-1}{i}(x^i+x^{i+p})-\sum_{i=0}^{2p-1}\binom{2p-1}{i}x^i\bigg]\\
&~~~+\frac{1}{p-1}\sum_{d\mid(p-1)}\Phi(d)\sum_{i=0}^{\frac{2(p-1)}{d}}\binom{\frac{2(p-1)}{d}}{i}(x^{di}+x^{di+1}).
\end{aligned} 
\end{equation*}
Note that the two Cayley digraphs $\mathrm{Cay}(D_{2p},S)$ and $\mathrm{Cay}(D_{2p},T)$ are isomorphic if and only if $S$ and $T$ are $\mathrm{Aut}(D_{2p})$-equivalent by Lemma \ref{DCI}. Thus the number of Cayley digraphs  on $D_{2p}$  of out-degree $k$ up to isomorphism is equal to the number of $\mathrm{Aut}(D_{2p})$-equivalent $k$-subsets of $A$, which is also the coefficient of $x^k$ in the polynomial $Q(x)$ by Lemma \ref{polya-2}.  If $1\leq k\leq p-1$, the coefficient of $x^k$ in the polynomial $Q(x)$ is given by 
\begin{equation*}
\frac{1}{p}\bigg[\binom{p-1}{k}-\binom{2p-1}{k}\bigg]+\frac{1}{p-1}\bigg[\sum_{d\mid\mathrm{gcd}(p-1,k)}\Phi(d)\binom{\frac{2(p-1)}{d}}{\frac{k}{d}}+\sum_{d\mid\mathrm{gcd}(p-1,k-1)}\Phi(d)\binom{\frac{2(p-1)}{d}}{\frac{k-1}{d}}\bigg].
\end{equation*}
Similarly, if $p\leq k\leq 2p-2$, the coefficient of $x^k$ in the polynomial $Q(x)$ is equal to
\begin{equation*}
\frac{1}{p}\bigg[\binom{p-1}{k-p}-\binom{2p-1}{k}\bigg]+\frac{1}{p-1}\bigg[\sum_{d\mid\mathrm{gcd}(p-1,k)}\Phi(d)\binom{\frac{2(p-1)}{d}}{\frac{k}{d}}+\sum_{d\mid\mathrm{gcd}(p-1,k-1)}\Phi(d)\binom{\frac{2(p-1)}{d}}{\frac{k-1}{d}}\bigg].
\end{equation*}
We obtain the  result as required.
\end{proof}
\begin{lem}(See \cite{Mishna}.)\label{circ-di-degree}
Let $p$ be an odd prime. The number of circulant digraphs of order $p$ with out-degree $k$ up to isomorphism is given by
\begin{equation}\label{eq-12-3}
\mathcal{M}_{c,k}=\frac{1}{p-1}\sum_{d\mid\mathrm{gcd}(p-1,k)}\Phi(d)\binom{\frac{p-1}{d}}{\frac{k}{d}},
\end{equation}
where $\Phi(\cdot)$ is the Euler's totient function.
\end{lem}
Let $\mathrm{Cay}(D_{2p},S)$ be a Cayley graph on $D_{2p}$ with $|S|=k$. If $k=0$ or $1$, then $\mathrm{Cay}(D_{2p},S)$ is obviously disconnected. Observe that for $2\leq k\leq p-1$, $\mathrm{Cay}(D_{2p},S)$ is disconnected if and only if $S\subseteq A_1=\langle\tau\rangle\setminus\{e\}=\{\tau^i\mid i\in \mathbb{Z}_p\setminus\{0\}\}$, and for $p\leq k\leq 2p-1$, $\mathrm{Cay}(D_{2p},S)$ must be connected. By Theorem \ref{thm-enu-3} and Lemma \ref{circ-di-degree}, we obtain the following result.
\begin{thm}\label{thm-enu-3-1}
Let $p$ be an odd prime. The number of connected Cayley digraphs on $D_{2p}$  of out-degree $k$ up to isomorphism is
\begin{equation}\label{eq-12-4}
\mathcal{M}_k'=
\left\{
\begin{array}{ll}
0 &\mbox{if } k=0,1,\\ 
\mathcal{M}_k-\mathcal{M}_{c,k} &\mbox{if } 2\leq k\leq p-1,\\
\mathcal{M}_k & \mbox{if } p\leq k\leq 2p-1,
\end{array}
\right.
\end{equation}
where $\mathcal{M}_k$ and $\mathcal{M}_{c,k}$ are presented in (\ref{eq-12-2}) and (\ref{eq-12-3}), respictively.
\end{thm}
\begin{exam}\label{exam-1}
\emph{Take $p=3$ and consider the dihedral group $D_6=\langle\tau,\sigma\mid\tau^3=\sigma^2=e\rangle=\{\tau^i,\tau^j\sigma\mid 0\leq i,j\leq 2\}$.
Let $A=D_6\setminus\{e\}$. Then it is easy to see that all the representative elements of $\mathrm{Aut}(D_6)$-equivalent classes of subsets of $A$ are as follows: $\emptyset$, $\{\tau\}$, $\{\sigma\}$, $\{\tau,\tau^2\}$,  $\{\tau,\sigma\}$, $\{\sigma,\tau\sigma\}$,  $\{\tau,\sigma,\tau\sigma\}$, $\{\tau,\tau^2,\sigma\}$, $\{\sigma,\tau\sigma,\tau^2\sigma\}$, $\{\tau,\sigma,\tau\sigma,\tau^2\sigma\}$, $\{\tau,\tau^2,\sigma,\tau\sigma\}$, $\{\tau,\tau^2,\sigma,\tau\sigma,\tau^2\sigma\}$. Thus there are exactly  twelve Cayley digraphs on $D_6$ up to isomorphism in which eight are connected. By Theorems \ref{thm-enu-1} and \ref{thm-enu-2}, 
\begin{equation*}
\begin{aligned}
\mathcal{N}&=\frac{1}{3}\cdot2^{3}(1-2^{2})+\sum_{d\mid2}\Phi(d)2^{\frac{4}{d}}=12,\\
\mathcal{N}_1&=\mathcal{N}-\mathcal{N}_c-1=12-\frac{1}{2}\sum_{d\mid 2}\Phi(d)2^{\frac{2}{d}}-1=8,
\end{aligned}
\end{equation*}
as required. Also, we see that there are exactly  three Cayley digraphs on $D_6$ with out-degree $k=2$ up to isomorphism in which two are connected. By Theorems \ref{thm-enu-3} and \ref{thm-enu-3-1}, we obtain 
\begin{equation*}
\begin{aligned}
\mathcal{M}_2&=\frac{1}{3}\bigg[\binom{2}{2}-\binom{5}{2}\bigg]+\frac{1}{2}\bigg[\sum_{d\mid\mathrm{gcd}(2,2)}\Phi(d)\binom{\frac{4}{d}}{\frac{2}{d}}+\sum_{d\mid\mathrm{gcd}(2,1)}\Phi(d)\binom{\frac{4}{d}}{\frac{1}{d}}\bigg]=3,\\
\mathcal{M}_2'&=\mathcal{M}_2-\mathcal{M}_{c,2}=3-\frac{1}{2}\sum_{d\mid\mathrm{gcd}(2,2)}\Phi(d)\binom{\frac{2}{d}}{\frac{2}{d}}=2.
\end{aligned}
\end{equation*}
}
\end{exam}
\begin{exam}\label{exam-2}
\emph{
Combining Theorems \ref{thm-enu-3} and \ref{thm-enu-3-1}, we have
\begin{equation}\label{eq-12-5}
\sum_{k=2}^{2p-1}\mathcal{M}_k'=\mathcal{N}_1.
\end{equation} 
In Tab. \ref{tab-1}, we list the number of connected Cayley digraphs on $D_{2p}$ with out-degree $k$  up to isomorphism for each $2\leq k\leq 2p-1$, where $3\leq p\leq 19$.  One can easily verify that (\ref{eq-12-5}) holds for each $p$.
\begin{table}[h]\footnotesize{
\caption{\label{tab-1}
\small{The number of connected Cayley digraphs on $D_{2p}$ ($3\leq p\leq 19$).}}
\begin{tabular*}{15cm}{@{\extracolsep{\fill}}c|c|c}
\toprule
$p$&$(\mathcal{M}_2',\ldots,\mathcal{M}_{2p-1}')$&$\mathcal{N}_1$ \\
  \midrule
3&$(2,3,2,1)$&8\\
\hline
5&$(2,5,8,9,6,4,2,1)$&37\\
\hline
7&$(2,8,20,33,46,47,34,23,12,5,2,1)$& 233\\
\hline
11&$\begin{array}{c}(2,12,56,188,504,1068,1870,2693,3234,3235,2694,1875,\\1080,526,214,78,24,7,2,1)\end{array}$&19363\\
\hline
13&$\begin{array}{c}(2,15,84,342,1150,3100,6972,13147,21028,28639,33428,33429,\\28640,21034,13166,7015,3166,1230,408,127,34,8,2,1)\end{array}$&216167\\
\hline
17&$\begin{array}{c}(2,19,150,865,4074,15699,51110,141867,340538,	711789,	1304974,	2107735,\\3011042,3813797,4290532,4290533,3813798,3011050,2107770,1305090,\\712062,341042,142582,51920,16414,4578,1138,	266,	54,10,2,1)\end{array}$&31592503\\
\hline
19&$\begin{array}{c}(2,22,192,1258,6790,30064,112924,363804,1018876,2500360,5417644,\\
10417634,17858736,27382554,37650982,46509357,51677170,51677171,\\
46509358,37650991,27382600,17858908,10418110,5418682,2502128,\\
1021314,366508,115362,31832,7828,1734,364,68,11,2,1)\end{array}$&401911341\\
\bottomrule
\end{tabular*}}
\end{table}
}
\end{exam}

\section{Enumerating Cayley graphs on $D_{2p}$}\label{s-4}
Let $D_{2p}=\langle\tau,\sigma\mid \tau^p=\sigma^2=e, \sigma\tau\sigma=\tau^{-1}\rangle=\{\tau^i,\tau^j\sigma\mid i,j \in \mathbb{Z}_p\}$ be the dihedral group of order $2p$  ($p$ is an odd prime).  Take $A=A_1\cup A_2$, where
\begin{equation}\label{eq-13}
A_1=\Big\{\bar{\tau}^i=\{\tau^i,\tau^{-i}\}\mid 1\leq i\leq \frac{p-1}{2}\Big\}\mbox{ and } A_2=\{\tau^j\sigma\mid j\in\mathbb{Z}_p\}.
\end{equation}
Then $\mathrm{Aut}(D_{2p})$ acts on $A$ naturally by setting $\alpha_{s,t}(\bar{\tau}^i)=\bar{\tau}^{si}$ and $\alpha_{s,t}(\tau^j\sigma)=\tau^{sj+t}\sigma$ for each $\alpha_{s,t}\in \mathrm{Aut}(D_{2p})$.  Furthermore, if $\alpha_{s,t}\in \mathrm{Aut}(D_{2p})$ fixes every element of $A$, then $\alpha_{s,t}(\tau^j\sigma)=\tau^j\sigma$ for each $j\in\mathbb{Z}_p$. Take $j=0$, we have  $\alpha_{s,t}(\tau^0\sigma)=\tau^t\sigma=\tau^0\sigma$, which implies that $t=0$. Take some $j=j_0\in\mathbb{Z}_p\setminus\{0\}$, then $\alpha_{s,t}(\tau^{j_0}\sigma)=\alpha_{s,0}(\tau^{j_0}\sigma)=\tau^{sj_0}\sigma=\tau^{j_0}\sigma$, which gives that $sj_0=j_0$ in $\mathbb{Z}_p$, thus we must have $s=1$ due to $j_0\in\mathbb{Z}_p\setminus\{0\}=\mathbb{Z}_p^\times$ is invertible. Therefore, we obtain that $\alpha_{s,t}=\alpha_{1,0}$, which is the identity element of $\mathrm{Aut}(D_{2p})$. This implies that $\mathrm{Aut}(D_{2p})$ acts on $A$ faithfully, and so can be viewed as a permutation group on $A$. Also  observe that $\alpha_{s,t}(A_1)=A_1$ and $\alpha_{s,t}(A_2)=A_2$ for each $\alpha_{s,t}\in \mathrm{Aut}(D_{2p})$. Let $C=\{0,1\}$. As explained at the begining of Section \ref{s-3}, we may conclude that the number of Cayley graphs on $D_{2p}$ up to isomorphism is equal to the number of orbits of the group action $(\mathrm{Aut}(D_{2p}),C^A)$. Thus the primary task is to calculate the cycle index of $\mathrm{Aut}(D_{2p})$ acting on $A$.

\begin{lem}\label{cycle-index-3}
Let $A=A_1\cup A_2$ be defined as in (\ref{eq-13}), and let $\mathrm{Aut}(D_{2p})$ act on $A$ as above. Let $z$ be a generating element of the cyclic group $\mathbb{Z}_p^\times$. Under the action of $\mathrm{Aut}(D_{2p})$ on $A$, the cycle type of $\alpha_{s,t}\in \mathrm{Aut}(D_{2p})$ is given by  $type(\alpha_{s,t})=(b_1(\alpha_{s,t}),b_2(\alpha_{s,t}),\ldots,b_{\frac{3p-1}{2}}(\alpha_{s,t}))$, where 
\begin{equation}\label{eq-14}
b_k(\alpha_{1,t})=\left\{
\begin{aligned}
&\frac{3p-1}{2}&~&\mbox{if }k=1\mbox{ and } t=0,\\
&\frac{p-1}{2}&~&\mbox{if }k=1\mbox{ and } t\in\mathbb{Z}_p\setminus\{0\},\\
&1&~&\mbox{if } k=p \mbox{ and }t\in\mathbb{Z}_p\setminus\{0\},\\
&0&~&\mbox{otherwise},
\end{aligned}
\right.
\end{equation}
and for each $1\neq s=z^{i_s}\in\mathbb{Z}_p^\times$ (i.e., $i_s\neq 0$) and $t\in\mathbb{Z}_p$,
\begin{equation}\label{eq-15}
b_k(\alpha_{s,t})=b_k(\alpha_{s,0})=\left\{
\begin{aligned}
&1&~&\mbox{if }k=1,\\
&\frac{\mathrm{gcd}(2i_s,p-1)}{2}&~&\mbox{if } k=\frac{p-1}{\mathrm{gcd}(2i_s,p-1)},\\
&\mathrm{gcd}(i_s,p-1)&~&\mbox{if } k=\frac{p-1}{\mathrm{gcd}(i_s,p-1)},\\
&0&~&\mbox{otherwise}.
\end{aligned}
\right.
\end{equation}
(Note that if $s=-1$, i.e., $i_s=\frac{p-1}{2}$,  then $\frac{p-1}{\mathrm{gcd}(2i_s,p-1)}=1$, and so $b_1(\alpha_{-1,t})=b_1(\alpha_{-1,0})=1+\frac{\mathrm{gcd}(2i_s,p-1)}{2}=\frac{p+1}{2}$ for each $t\in \mathbb{Z}_p$.)
\end{lem}
\begin{proof}
Since $\alpha_{s,t}(A_1)=A_1$ and $\alpha_{s,t}(A_2)=A_2$ for each $\alpha_{s,t}\in \mathrm{Aut}(D_{2p})$, we must have $b_{p+1}(\alpha_{s,t})=\dots=b_{\frac{3p-1}{2}}(\alpha_{s,t})=0$. For $\alpha_{s,t}\in \mathrm{Aut}(D_{2p})$, we consider the following two situations.

\textit{Case 1.} $s=1$;

Note that $\alpha_{1,t}(\bar{\tau}^i)=\bar{\tau}^i$ for $1\leq i\leq \frac{p-1}{2}$, so the permutation $\alpha_{1,t}$ splits $A_1$ into $\frac{p-1}{2}$ cycles each of length $1$.  Also note that $\alpha_{1,t}(\tau^j\sigma)=\tau^{j+t}\sigma$ for $j\in \mathbb{Z}_p$. As in the proof of Lemma \ref{cycle-index-1}, we may conclude that the permutation $\alpha_{1,t}$ splits $A_2$ into $p$ cycles each of length $1$ when $t=0$ and exactly one cycle of length $p$ when $t\in \mathbb{Z}_p\setminus\{0\}$. Therefore, we have obtained the cycle type of $\alpha_{1,t}$, as shown in (\ref{eq-14}).

\textit{Case 2.} $s\neq1$, say $s=z^{i_s}$ ($i_s\neq 0$);

Note that $\alpha_{s,t}(\bar{\tau}^i)=\bar{\tau}^{si}$ for $i\in \mathbb{Z}_p\setminus\{0\}$ and $\alpha_{s,t}(\tau^j\sigma)=\tau^{sj+t}\sigma$ for $j\in \mathbb{Z}_p$. Firstly, we shall prove that  $\alpha_{s,t}$ has the same cycle type as $\alpha_{s,0}$ for each $t\in\mathbb{Z}_p$. As $\alpha_{s,t}(\bar{\tau}^i)=\bar{\tau}^{si}=\alpha_{s,0}(\bar{\tau}^i)$, $\alpha_{s,t}$ and $\alpha_{s,0}$ have the same cycle type in $A_1$. Now consider $\alpha_{s,t}$ and $\alpha_{s,0}$ acting on $A_2$. By the same method as in the proof of Lemma \ref{cycle-index-2}, one can deduce that  $\alpha_{s,t}$ and $\alpha_{s,0}$ also have the same cycle type in $A_2$. Hence, we just need to consider the cycle type of $\alpha_{s,0}$ in $A=A_1\cup A_2$. For any fixed $\bar{\tau}^i\in A_1$, assume that $\bar{o}(s)$ is the minimal positive integer such that $\alpha_{s,0}^{\bar{o}(s)}(\bar{\tau}^i)=\bar{\tau}^{s^{\bar{o}(s)}\cdot i}=\bar{\tau}^{z^{i_s\bar{o}(s)}\cdot i}=\bar{\tau}^i$. Then we have $\tau^{z^{i_s\bar{o}(s)}\cdot i}=\tau^i$ and $\tau^{-z^{i_s\bar{o}(s)}\cdot i}=\tau^{-i}$, or  $\tau^{z^{i_s\bar{o}(s)}\cdot i}=\tau^{-i}$ and $\tau^{-z^{i_s\bar{o}(s)}\cdot i}=\tau^{i}$. For the former, we obtain that   $z^{i_s\bar{o}(s)}=1$ and so $i_s\bar{o}(s)\equiv 0~(\mathrm{mod}~p-1)$, which gives that $i_s\bar{o}(s)\equiv 0~(\mathrm{mod}~\frac{p-1}{2})$; for the later, we get $z^{i_s\bar{o}(s)}=-1=z^{\frac{p-1}{2}}$ and so $i_s\bar{o}(s)\equiv \frac{p-1}{2}~(\mathrm{mod}~p-1)$, which also implies that $i_s\bar{o}(s)\equiv 0~(\mathrm{mod}~\frac{p-1}{2})$. By the minimality of $\bar{o}(s)$, we claim that $\bar{o}(s)=\frac{\frac{p-1}{2}}{\mathrm{gcd}(i_s,\frac{p-1}{2})}=\frac{p-1}{\mathrm{gcd}(2i_s,p-1)}$. Therefore, $\bar{\tau}^i\in A_1$ is in the cycle $(\bar{\tau}^i,\alpha_{s,0}(\bar{\tau}^i),\alpha_{s,0}^2(\bar{\tau}^i),\ldots,\alpha_{s,0}^{\bar{o}(s)-1}(\bar{\tau}^i))=(\bar{\tau}^i,\bar{\tau}^{si},\bar{\tau}^{s^2i},\ldots,$ $\bar{\tau}^{s^{\bar{o}(s)-1}i})$. Thus the permutation $\alpha_{s,0}$ splits $A_1$ into $\frac{\frac{p-1}{2}}{\bar{o}(s)}=\frac{\mathrm{gcd}(2i_s,p-1)}{2}$ cycles each of length $\bar{o}(s)=\frac{p-1}{\mathrm{gcd}(2i_s,p-1)}$. Also note that $\bar{o}(s)=1$ if and only if $s=-1$ due to $s\neq 1$. 

 Since the order of $s$ in $\mathbb{Z}_p^\times$ is equal to $o(s)=o(z^{i_s})=\frac{p-1}{\mathrm{gcd}(i_s,p-1)}$,  for any $j\in \mathbb{Z}_p\setminus\{0\}$, 
 $\tau^j\sigma\in A_2$ is in the cycle $(\tau^j\sigma,\alpha_{s,0}(\tau^j\sigma),\alpha_{s,0}^2(\tau^j\sigma),\ldots,$ $\alpha_{s,0}^{o(s)-1}(\tau^j\sigma))=(\tau^j\sigma,\tau^{sj}\sigma,\tau^{s^2j}\sigma,\ldots,\tau^{s^{o(s)-1}j}\sigma)$. Also note that $\alpha_{s,0}(\tau^0\sigma)=\tau^{0}\sigma$, so $\tau^0\sigma\in A_2$ is in the cycle $(\tau^0\sigma)$. Thus the permutation $\alpha_{s,0}$ splits  $A_2$ into $\frac{p-1}{o(s)}=\mathrm{gcd}(i_s,p-1)$ cycles each of length $o(s)=\frac{p-1}{\mathrm{gcd}(i_s,p-1)}$ and one cycle of length $1$. Then we have obtained the cycle type of $\alpha_{s,t}$, as shown in (\ref{eq-15}).
 
 We complete the proof.
\end{proof}

\begin{lem}\label{cycle-index-4}
Let $A=A_1\cup A_2$ be defined as in (\ref{eq-13}). The cycle index of $\mathrm{Aut}(D_{2p})$ acting on $A$ is given by
\begin{equation}\label{eq-16}
\mathcal{I}(\mathrm{Aut}(D_{2p}),A)=\frac{1}{p}x_1^{\frac{p-1}{2}}(x_p-x_1^p)+\frac{1}{p-1}x_1\cdot\sum_{i=0}^{p-2}x_{\frac{p-1}{\mathrm{gcd}(2i,p-1)}}^{\frac{\mathrm{gcd}(2i,p-1)}{2}}x_{\frac{p-1}{\mathrm{gcd}(i,p-1)}}^{\mathrm{gcd}(i,p-1)}.
\end{equation}
\end{lem}
\begin{proof}
By Lemma \ref{cycle-index-3}, the cycle index of $\mathrm{Aut}(D_{2p})$ acting on $A$ is 
\begin{eqnarray*}
&&\mathcal{I}(\mathrm{Aut}(D_{2p}),A)\\
&=&\frac{1}{|\mathrm{Aut}(D_{2p})|}\sum_{\alpha_{s,t}\in \mathrm{Aut}(D_{2p})}{x_1^{b_1(\alpha_{s,t})}x_2^{b_2(\alpha_{s,t})}\cdots x_{\frac{3p-1}{2}}^{b_{\frac{3p-1}{2}}(\alpha_{s,t})}}\\
&=&\frac{1}{p(p-1)}\sum_{s\in\mathbb{Z}_p^\times}\sum_{t\in\mathbb{Z}_p}{x_1^{b_1(\alpha_{s,t})}x_2^{b_2(\alpha_{s,t})}\cdots x_{\frac{3p-1}{2}}^{b_{\frac{3p-1}{2}}(\alpha_{s,t})}}\\
&=&\frac{1}{p(p-1)}\bigg[\sum_{t\in\mathbb{Z}_p}x_1^{b_1(\alpha_{1,t})}x_2^{b_2(\alpha_{1,t})}\cdots x_{\frac{3p-1}{2}}^{b_{\frac{3p-1}{2}}(\alpha_{1,t})}+\sum_{s\in\mathbb{Z}_p^\times\setminus\{1\}}\sum_{t\in\mathbb{Z}_p}x_1^{b_1(\alpha_{s,t})}x_2^{b_2(\alpha_{s,t})}\cdots x_{\frac{3p-1}{2}}^{b_{\frac{3p-1}{2}}(\alpha_{s,t})}\bigg]\\
&=&\frac{1}{p(p-1)}\bigg[\sum_{t\in\mathbb{Z}_p}x_1^{b_1(\alpha_{1,t})}x_2^{b_2(\alpha_{1,t})}\cdots x_{\frac{3p-1}{2}}^{b_{\frac{3p-1}{2}}(\alpha_{1,t})}+p\cdot\sum_{s\in\mathbb{Z}_p^\times\setminus\{1\}}x_1^{b_1(\alpha_{s,0})}x_2^{b_2(\alpha_{s,0})}\cdots x_{\frac{3p-1}{2}}^{b_{\frac{3p-1}{2}}(\alpha_{s,0})}\bigg]\\
&=&\frac{1}{p(p-1)}\bigg[x_1^{\frac{3p-1}{2}}+(p-1)x_1^{\frac{p-1}{2}}x_p+px_1\cdot\sum_{s=z^{i_s}\in\mathbb{Z}_p^\times\setminus\{1\}}x_{\frac{p-1}{\mathrm{gcd}(2i_s,p-1)}}^{\frac{\mathrm{gcd}(2i_s,p-1)}{2}}x_{\frac{p-1}{\mathrm{gcd}(i_s,p-1)}}^{\mathrm{gcd}(i_s,p-1)}\bigg]\\
&=&\frac{1}{p(p-1)}\bigg[x_1^{\frac{3p-1}{2}}+(p-1)x_1^{\frac{p-1}{2}}x_p+px_1\cdot\sum_{i_s=1}^{p-2}x_{\frac{p-1}{\mathrm{gcd}(2i_s,p-1)}}^{\frac{\mathrm{gcd}(2i_s,p-1)}{2}}x_{\frac{p-1}{\mathrm{gcd}(i_s,p-1)}}^{\mathrm{gcd}(i_s,p-1)}\bigg]\\
&=&\frac{1}{p(p-1)}\bigg[x_1^{\frac{3p-1}{2}}+(p-1)x_1^{\frac{p-1}{2}}x_p+px_1\cdot\sum_{i_s=0}^{p-2}x_{\frac{p-1}{\mathrm{gcd}(2i_s,p-1)}}^{\frac{\mathrm{gcd}(2i_s,p-1)}{2}}x_{\frac{p-1}{\mathrm{gcd}(i_s,p-1)}}^{\mathrm{gcd}(i_s,p-1)}-px_1^{\frac{3p-1}{2}}\bigg]\\
&=&\frac{1}{p}x_1^{\frac{p-1}{2}}(x_p-x_1^p)+\frac{1}{p-1}x_1\cdot\sum_{i=0}^{p-2}x_{\frac{p-1}{\mathrm{gcd}(2i,p-1)}}^{\frac{\mathrm{gcd}(2i,p-1)}{2}}x_{\frac{p-1}{\mathrm{gcd}(i,p-1)}}^{\mathrm{gcd}(i,p-1)}.
\end{eqnarray*}
\end{proof}
According to Lemmas \ref{cycle-index-4} and \ref{polya-1}, we obtain the number of Cayley graphs on $D_{2p}$ up to isomorphism immediately.
\begin{thm}\label{thm-enu-4}
Let $p$ be an odd prime. The number of Cayley graphs on $D_{2p}$ up to isomorphism is equal to
\begin{equation}\label{eq-17}
\mathcal{N}'=\frac{1}{p}(2^{\frac{p+1}{2}}-2^{\frac{3p-1}{2}})+\frac{2}{p-1}\sum_{i=0}^{p-2}2^{\frac{\mathrm{gcd}(2i,p-1)}{2}+\mathrm{gcd}(i,p-1)}.
\end{equation}
\end{thm}
In \cite{Mishna}, Mishna also enumerated the circulant digraphs of order $p$ up to isomorphism.
\begin{lem}(See \cite{Mishna}.)\label{circ-graph}
Let $p$ be an odd prime. The number of circulant  graphs of order $p$  up to isomorphism is equal to
\begin{equation}\label{eq-18}
\mathcal{N}_{c}'=\frac{2}{p-1}\sum_{d\mid\frac{p-1}{2}}\Phi(d)2^{\frac{p-1}{2d}},
\end{equation}
where $\Phi(\cdot)$ is the Euler's totient function.
\end{lem}
As in Section \ref{s-3}, from Theorem \ref{thm-enu-4} and Lemma \ref{circ-graph} we also give the number of  connected Cayley graphs on $D_{2p}$ up to isomorphism.
\begin{thm}\label{thm-enu-5}
Let $p$ be an odd prime. The number of connected Cayley graphs on $D_{2p}$ up to isomorphism is equal to
\begin{equation}\label{eq-19}
\mathcal{N}_1'=\mathcal{N}'-\mathcal{N}_{c}'-1,
\end{equation}
where $\mathcal{N}'$ and $\mathcal{N}_{c}'$ are presented in (\ref{eq-17}) and (\ref{eq-18}), respictively.
\end{thm}
\begin{exam}\label{exam-3}
\emph{As in Example \ref{exam-1}, we take $p=3$ and consider the dihedral group $D_6$.
Let $A=D_6\setminus\{e\}$. Then it is easy to see that all the representative elements of $\mathrm{Aut}(D_6)$-equivalent classes of inverse-closed subsets of $A$ are as follows: $\emptyset$, $\{\sigma\}$, $\{\tau,\tau^2\}$,  $\{\sigma,\tau\sigma\}$,  $\{\tau,\tau^2,\sigma\}$, $\{\sigma,\tau\sigma,\tau^2\sigma\}$,  $\{\tau,\tau^2,\sigma,\tau\sigma\}$, $\{\tau,\tau^2,\sigma,\tau\sigma,\tau^2\sigma\}$. Thus there are exactly  eight Cayley graphs on $D_6$ up to isomorphism in which five are connected. By Theorems \ref{thm-enu-4} and \ref{thm-enu-5}, we have
\begin{equation*}
\begin{aligned}
\mathcal{N}'&=\frac{1}{3}(2^{\frac{4}{2}}-2^{\frac{8}{2}})+\sum_{i=0}^{1}2^{\frac{\mathrm{gcd}(2i,2)}{2}+\mathrm{gcd}(i,2)}=8,\\
\mathcal{N}_1'&=\mathcal{N}'-\mathcal{N}_c'-1=8-\sum_{d\mid 1}\Phi(d)2^{\frac{2}{2d}}-1=5.
\end{aligned}
\end{equation*}
}
\end{exam}

At the end of this paper, we list the number of connected Cayley graphs on $D_{2p}$   ($p$ prime) up to isomorphism for $3\leq p\leq 41$ by applying Theorem \ref{thm-enu-5} (see Tab. \ref{tab-2}).

\begin{table}[h]\footnotesize{
\caption{\label{tab-2}
\small{The number of connected Cayley graphs on $D_{2p}$ ($3\leq p\leq 41$).}}
\begin{tabular*}{15cm}{@{\extracolsep{\fill}}cccccc}
\toprule
$p$&$\mathcal{N}_1'$&$p$& $\mathcal{N}_1'$&$p$& $\mathcal{N}_1'$\\
  \midrule
  3&5&13&4089&29&10851803161\\
  5&16&17&131623&31&75736903695\\
  7&51&19&814155&37&27052464217661\\
  11&807&23&34333611&41&1406056810572383\\
\bottomrule
\end{tabular*}}
\end{table}

\end{document}